\documentclass[12pt]{amsart}
\usepackage{amsmath,amsthm,amsfonts,amssymb,latexsym}

\begin{document}

\theoremstyle{plain}

\newtheorem{thm}{Theorem}[section]
\newtheorem{lem}[thm]{Lemma}
\newtheorem{pro}[thm]{Proposition}
\newtheorem{cor}[thm]{Corollary}
\newtheorem{que}[thm]{Question}
\newtheorem{rem}[thm]{Remark}
\newtheorem{defi}[thm]{Definition}

\newtheorem*{thmA}{Theorem A}
\newtheorem*{corB}{Corollary B}

\newcommand{\PSL}{\operatorname{PSL}}
\newcommand{\PSU}{\operatorname{PSU}}
\newcommand{\AAA}{\sf A}
\newcommand{\Aut}{\operatorname{Aut}}

\def\Z{{\mathbb Z}}
\def\C{{\mathbb C}}
\def\Q{{\mathbb Q}}
\def\F{{\mathbb F}}
\def\irr#1{{\rm Irr}(#1)}
\def\irrv#1{{\rm Irr}_{\rm rv}(#1)}
\def \c#1{{\cal #1}}
\def\cent#1#2{{\bf C}_{#1}(#2)}
\def\syl#1#2{{\rm Syl}_#1(#2)}
\def\nor{\triangleleft\,}
\def\oh#1#2{{\bf O}_{#1}(#2)}
\def\Oh#1#2{{\bf O}^{#1}(#2)}
\def\zent#1{{\bf Z}(#1)}
\def\zentt#1{{\bf Z}_2(#1)}
\def\det#1{{\rm det}(#1)}
\def\ker#1{{\rm ker}(#1)}
\def\norm#1#2{{\bf N}_{#1}(#2)}
\def\alt#1{{\rm Alt}(#1)}
\def\iitem#1{\goodbreak\par\noindent{\bf #1}}
\def \mod#1{\, {\rm mod} \, #1 \, }
\def\sbs{\subseteq}
\def\exp#1{{\rm exp}(#1)}

\def\gc{{\bf GC}}
\def\nlgc{{\bf GC^{*}}}

\def\o#1{\overline{#1}}

\newcommand{\gen}[1]{\left < #1 \right >}
\newcommand{\Gal}{{\it Gal}}

\def\skipa{\vspace{-1.5mm} & \vspace{-1.5mm} & \vspace{-1.5mm}\\}
\newcommand{\tw}[1]{{}^#1\!}
\renewcommand{\mod}{\bmod \,}
\renewcommand{\phi}{\varphi}

\marginparsep-0.5cm

\renewcommand{\thefootnote}{\fnsymbol{footnote}}
\footnotesep6.5pt

\title[Characters of the same degree are Galois conjugate]{Finite groups whose non-linear irreducible characters of the same degree are Galois conjugate}
\author{Silvio Dolfi}
\address{Dipartimento di Matematica, Universit\`a di Firenze, 50134 Firenze, Italy}
\email{dolfi@math.unifi.it}
\author{Manoj K. Yadav}
\address{School of Mathematics, Harish-Chandra Research Institute, Chhatnag Road, Jhunsi, Allahabad - 211019, India}
\email{myadav@hri.res.in}

\begin{abstract}
We classify the finite groups whose non-linear irreducible characters that are not conjugate under the natural Galois action have distinct degrees, therefore extending the results in Berkovich et al. [Proc. Amer. Math. Soc. {\bf 115} (1992), 955-959] and Dolfi et al. [Israel J. Math. {\bf 198} (2013), 283-331].
\end{abstract}

\thanks{The first author thanks the Harish-Chandra Research Institute of Allahabad for 
hospitality in February 2014 and April 2015; he has been partially supported by the Italian INdAM-GNSAGA}

\maketitle

\section{Introduction}
In  1992,  Berkovich,  Chillag and  Herzog \cite{BCH}  classified the 
finite groups whose non-linear irreducible characters all have distinct degrees. 
Since Galois groups of suitable cyclotomic fields act in a natural degree-preserving way  (see below) on the set 
$\irr G$ of the irreducible characters of a finite group $G$, it seems natural to weaken the above mentioned condition
by  asking that there exists just one orbit on $\irr G$   for every given irreducible 
character degree $\neq 1$. 
While the condition in~\cite{BCH} forces all non-linear characters in $\irr G$ to be rational 
valued, we are now just imposing a minimality condition on the multiplicities of the 
degrees of the irreducible characters, without setting restrictions on their fields of values. 

Let $G$ be a finite group, $n$ a multiple of $|G|$ and let $\mathcal{G}_n={\rm Gal}(\Q_n|\Q)$
be the Galois group of the $n$-th cyclotomic extension.
Then $\mathcal G_n$ acts on the set $\irr G$ 
as follows: for $\alpha \in \mathcal{G}_n$, $\chi \in \irr G$ and $g \in G$, we define
$$\chi^{\alpha}(g) = \chi(g)^{\alpha} \; . $$

For $\chi, \psi \in \irr G$, if there exists a Galois automorphism 
$\alpha \in \mathcal{G}_n$ such that $\chi^\alpha = \psi$, then  we say that
$\chi$ and $\psi$ are \emph{Galois conjugate} (in $\mathcal{G}_n$).
This is clearly an equivalence relation on $\irr G$. Characters in the same equivalence class 
have the same kernel, center, field of values and degree.

In this paper, we weaken the condition of \cite{BCH}, and prove the following result.

\begin{thmA}
  \label{structure}
Let $G$ be a finite group. Every two non-linear irreducible characters of the same degree of $G$ are Galois conjugate  if and only if $G$ is either abelian or one 
of the following. 
\begin{description}
\item[(a)] $G$ is a $p$-group ($p$ a prime), $|G'| = p$ and $\zent G$ is cyclic;  
\item[(b)] $G$ is a Frobenius group with prime power order kernel $K$ and
complement $L$, with $L$ cyclic or $L \cong Q_8$.  
Moreover: 
\begin{description}
\item[(b1)] $L \cong Q_8$ and $|K| = 3^2$; or
\item[(b2)] $K$ is elementary abelian,   $|K| = q^n$ ( $q$  prime), 
$L$ is cyclic and $|L| = (q^n -1)/d$, where $d$ divides $q-1$ and $(d,n) = 1$; or
\item[(b3)] $K$ is a Suzuki $2$-group with $|K| = |K'|^2$ and $L$ is cyclic of 
order  $|K'| -1$. 
\end{description}

\item[(c)] $G$ is non-solvable and either 
$$G \in \{ \AAA_5,Sz(8),J_2,J_3,L_3(2),M_{22},Ru,Th,{}^3D_4(2)\}$$
or 
$$G \in  \{\AAA_5 \times Sz(8),\AAA_5 \times Th, L_3(2) \times Sz(8)\} \; .$$ 
\end{description}
\end{thmA}

As a consequence of Theorem~A, we get a new proof of the main result of~\cite{BCH}. 

\begin{corB}
Let $G$ be a finite group.
Then, for every non-linear $\chi, \psi \in \irr G$, $\chi \neq \psi$ implies
$\chi(1) \neq \psi(1)$ if and only if $G$ is either abelian or one of the following groups: 
\begin{description}
\item[(a)] extraspecial $2$-groups; 
\item[(b)] $G = KL$ is a Frobenius group with elementary abelian kernel $K$, $|K| = q^n$
for a prime $q$, and either 
\begin{description}
\item[(b1)] $L \cong Q_8$ and $q^n = 3^2$; or
\item[(b2)] $L$ is cyclic of order $q^n -1$. 
\end{description}
\end{description}  
\end{corB}

In~\cite{DNT} the finite groups such that all $\emph{non-principal}$ irreducible characters
of the same degree are Galois conjugate are classified. 
We remark that for  non-solvable groups, the class of groups studied in~\cite{DNT} and the class we are considering here in fact coincide by Theorem~\ref{perfect}. 
However, the two classes
differ significantly in the case of nilpotent groups: while the nilpotent groups in~\cite{DNT} are just groups of prime order~\cite{DNT}, or the trivial group,
here  we have $p$-groups with cyclic center and
commutator subgroup of prime order, or abelian groups (see Corollary~\ref{nilpotent}).  

Finally, we remark that by quoting Theorem 4.1 of~\cite{DNT} (see Theorem~\ref{DNT4.1}) our work depends on the Classification of Finite Simple Groups.

\section{Preliminaries}
In the following, by ``group'' we always mean ``finite group''. 
We  use standard notation in character theory, as in ~\cite{Is}. 
Given a character $\chi \in \irr G$, we define
$\Q(\chi) = \Q[\{ \chi(g) \mid g \in G\}]$, the field generated by the values of $\chi$; 
$\Q(\chi)$ is called the \emph{field of values} of $\chi$. 
We stress here that two characters in $\irr G$ are Galois conjugate in some 
Galois group $\mathcal{G}_n$ if and only if they are Galois conjugate in 
${\rm Gal}(\Q(\chi)|\Q)$ (see Lemma~\ref{lem1}(a)). 
Therefore, we omit the explicit reference to a specific Galois extension
of $\Q$, and we  simply say ``Galois conjugate''.

\begin{defi}
We say that a finite group $G$ is a $\nlgc$-group 
(or $G \in \nlgc$)
if any  two \emph{non-linear} irreducible characters of $G$ are Galois conjugate whenever
they have the same degree.  
\end{defi}

The following lemma collects some basic facts, often used without explicit reference. 
In particular, part (d) shows that the class $\nlgc$ is stable by taking factor groups. 
 
\begin{lem}
\label{lem1}
Let $G$ be a finite group. Then the following hold true.
\begin{description}
\item[(a)] Let $\chi, \psi  \in \irr G$, 
and let $E = \Q_n$be  any cyclotomic field such that $\Q(\chi) \subseteq  E$. 
Then $\chi$ and $\psi$ are Galois conjugate in $E$ if and only if they are 
Galois conjugate in $\Q(\chi)$. 

\item[(b)] If $\chi, \psi \in \irr G$  are Galois conjugate, then $\chi(1) = \psi(1)$, 
$\Q(\chi) = \Q(\psi)$,  $\ker{\chi} = \ker{\psi}$ and $\zent{\chi} = \zent{\psi}$.

\item[(c)] Let $G = A \times B$, with $A$ non-abelian. 
If $G \in \nlgc$, then $B = B'$.

\item[(d)] Let $N$ be a normal subgroup of $G$. If $G \in \nlgc$, then 
$G/N \in \nlgc$.

\end{description}
\end{lem}

\begin{proof}

(a) Since ${\rm Gal}(E|\Q)$ is abelian,  $\Q(\chi)|\Q$ is a normal extension and the claim follows
by extending (resp. restricting)  $\Q$-automorphisms  to $E$ (resp. to $\Q(\chi)$). 

(b) These assertions follow directly from the definitions. 

(c) Let $\alpha \in \irr A$ with $\alpha(1) > 1$  and let $\beta\in \irr B$ with $\beta(1) = 1$. 
Then $\chi = \alpha \times 1_B$ and 
$\psi = \alpha \times \beta$ are  non-linear irreducible characters of $G$ and
they have the same degree. It follows that $B \leq \ker{\chi} = \ker{\psi}$, so
$\beta = 1_B$. Hence, $B' = B$.  

(d) Let $\chi, \psi \in \irr{G/N}$ be non-linear characters of the same degree. 
Then the same is true for their inflations $\chi_0, \psi_0 \in \irr G$; so they 
are Galois conjugate and (observing that $\Q(\chi) = \Q(\chi_0)$) the claim follows.
\end{proof}

Let $N$ be a normal subgroup of $G$ and let  $\lambda \in \irr N$. 
We denote by $\irr{G|\lambda} = \{ \chi \in \irr G \:|\: [\chi_N, \lambda] \neq 0 \}$
the set of the irreducible characters of $G$ lying above $\lambda$. 
If $\lambda$ is invariant in $G$ and $|\irr{G|\lambda}| = 1$ we say that
$\lambda$ is \emph{fully ramified} in $G/N$. In this case, if 
$\chi \in \irr G$ is the (only) character lying above $\lambda$, then $\chi(g) = 0$
for all $g \in G \setminus N$ and $|G/N| = (\chi(1)/\lambda(1))^2$ (see~\cite[Problem 6.3]{Is}).

\begin{lem}
  \label{fullyramified}
Let $P$ be a $p$-group such that $|P'| = p$, where $p$ is a prime. Let $Z = \zent P$.
Then the following statements hold true.
\begin{description}
\item[(a)] Every non-linear irreducible character of $P$ is a faithful character of  degree
$\sqrt{|P:Z|}$; 
\item[(b)] Every  non-trivial character 
$\lambda \in \Lambda := \{\lambda \in \irr Z \mid P'\not\leq \ker{\lambda}\}$
is fully ramified in $P/Z$. 
The map 
$$f: \Lambda \rightarrow  \{\chi \in \irr P \mid\chi(1) > 1 \} $$ 
such that  $f(\lambda) =  \chi_{\lambda}$,  where $\chi_{\lambda}$ is the unique irreducible character
of $G$ lying over $\lambda$,  is a bijection.  
\end{description}
 \end{lem}

 \begin{proof}
   This follows from Theorem~7.5 of~\cite{Hu1}. 
 \end{proof}

We also need a classical result on irreducible modules for abelian groups. 

\begin{lem}
  \label{3.10}
Let $V$ be a faithful irreducible $A$-module, $|V| = q^n$ ($q$ prime), for an abelian group $A$.
Then $A$ is cyclic and the semidirect product $V \rtimes A$ is isomorphic to a subgroup of 
the affine group ${\rm GF}(q^n)^{+} \rtimes {\rm GF}(q^n)^{\times}$. 
Moreover, if $U$ is any other faithful irreducible $A$-module of characteristic $q$, then $|U| = |V|$.
\end{lem}
\begin{proof}
It follows from~\cite[II.3.10]{H} (and its  proof).   
\end{proof}

Finally, we give a result that will be used in pinning down  the structure of 
nilpotent residuals (which will turn out to be Frobenius kernels) of $\nlgc$-groups. 

\begin{lem}
  \label{2.6}
Let $G \in \nlgc$ be a Frobenius group, with Frobenius  kernel  $K$ a $q$-group ($q$ prime).
Let $N \leq K$ be  normal in  $G$ and let $\lambda \in \irr N$ be a non-principal
$K$-invariant character.  Then the following statements hold true.
\begin{description}
\item[(a)] If $q = 2$ and $\theta_1, \theta_2 \in \irr{K|\lambda}$ are 
characters of the same degree, then there exists a Galois automorphism 
$\alpha \in Gal(\Q_{q^k}|\Q)$, where $q^k = \exp K$, such that 
$\theta_1^{\alpha} = \theta_2$. 

\item[(b)] If $K/N$ is abelian and $\exp K = q$, then $|\irr{K|\lambda}| = 1$ 
(i.e. $\lambda$ is fully ramified in $K$).  
\end{description}
\end{lem}

\begin{proof}
Write $G = KL$ with $L$ Frobenius complement. 
Let $\theta_1, \theta_2 \in \irr{K|\lambda}$ be such that $\theta_1(1) = \theta_2(1)$. As  $G$ is a Frobenius group,   $\theta_1^G$ and $\theta_2^G$ are
non-linear irreducible characters of the same degree of $G$. 
Hence, as $G \in \nlgc$ and  $\Q(\theta_i^G)  \subseteq  \Q(\theta_i) \subseteq  \Q_{q^k}$, where $q^k = \exp K$, 
there exists a Galois automorphism 
$\alpha \in Gal(\Q_{q^k}|\Q)$ such that 
$(\theta_1^G)^{\alpha} = \theta_2^G$. By Clifford theory, there exists an element $x \in L$ such
that 
\begin{equation}\label{key}
\theta_1^{\alpha} = \theta_2^x \; .
\end{equation} 
Thus, by restricting to $N$, we get that $\lambda^{\alpha} = \lambda^x$. 
Now, for every positive integer $m$, $\lambda^{\alpha^{m}}  = \lambda^{x^{m}}$ because
Galois conjugation and group conjugation commute. As any non-trivial element of $L$
fixes only the trivial character of $N$, we deduce that  
$o(x)$ divides $o(\alpha)$, so $o(x)$ divides 
$q^{k-1}(q-1)$. Since $|L|$ is coprime to $q$, we conclude that 
$o(x)$ divides $q-1$.

\medskip 
(a): As $o(x) \mid (q-1)$, if  $q = 2$ then by~(\ref{key})  we get  $\theta_1^{\alpha} = \theta_2$ and (a) is proved. 

\medskip
(b): Assume now that $\exp K = q$ is prime and that $K/N$ is abelian.  
By~\cite[Lemma 12.6]{MW}, there exists a (unique) subgroup $U$ with $N \leq U \leq K$ such
that every $\phi \in \irr{U|\lambda}$ extends $\lambda$ and is fully ramified in $K/U$. 
It follows that $|\irr{K|\lambda}| = |U/N|$ and that all characters in $\irr{K|\lambda}$ have the 
same degree. By~(\ref{key}) we deduce that the action of $\mathcal{G} \times L$ on 
$\irr{K|\lambda}$ (defined, for $\theta \in \irr{K|\lambda}$ and $(\alpha, x) \in \mathcal{G} \times L$, 
by $\theta^{(\alpha, x)} = (\theta^{\alpha})^x = (\theta^x)^{\alpha}$) is transitive on $\irr{K|\lambda}$. 
Since $|\irr{K|\lambda}| = |U/N|$ is a power of $q$ and $\mathcal{G} \times L$ is a $q'$-group, it 
follows that $|\irr{K|\lambda}| = 1$. 
\end{proof}


\section{$\nlgc$-groups}

\begin{thm}
  \label{p-groups}
Let $P$ be a non-abelian $p$-group, $p$ a prime. 
Then $P$ is a $\nlgc$-group if and only if $P$ has cyclic center and
commutator subgroup of prime order. 
\end{thm}

\begin{proof}
Assume first that $|P'| = p$ and that $Z = \zent P$ is cyclic. 
By Lemma~\ref{fullyramified} the map $f$ from the set 
$\Lambda = \{ \lambda \in \irr Z | P' \not\leq \ker{\lambda} \}$ onto 
$\{ \chi \in \irr P | \chi(1) > 1 \}$ such that $f(\lambda) =  \chi_{\lambda}$, where $\irr{G|\lambda} = \{ \chi_{\lambda} \}$, is a bijection. 
Note that $\Lambda$ is also the set of the faithful
characters of $Z$, as $Z$ is cyclic. 
Moreover $\Q(\chi_{\lambda}) = \Q(\lambda)$, as $\left( \chi_{\lambda} \right)_Z$ is a multiple of
$\lambda$ and $\chi(x) = 0$ for all $x \in P \setminus Z$. 
Let $|Z| = p^a$. If $\lambda \in \Lambda$, then $\Q(\lambda) = \Q_{p^a}$ and hence, 
writing $\mathcal{G} = {\rm Gal}(\Q(\lambda)|\Q) = {\rm Gal}(\Q_{p^a}|\Q)$, we have 
that $|\mathcal{G}| = |\Lambda|$. 
Since any element of $\Lambda $ is stabilized only by the trivial automorphism of $\mathcal{G}$, 
it follows that $\mathcal{G}$ acts transitively on $\Lambda$. 
Let now $\chi_1, \chi_2 \in \irr G$ be non-linear characters; then $\chi_1 = f(\lambda_1)$ 
and $\chi_2 = f(\lambda_2)$ for suitable $\lambda_1, \lambda_2 \in \Lambda$. 
Now, there exists a Galois automorphism $\alpha \in \mathcal{G}$ such that $\lambda_1^{\alpha} = \lambda_2$. 
As $(\chi_{\lambda_1})^{\alpha}$ lies over $\lambda_1^{\alpha}$, we have that 
$\chi_2 = f(\lambda_2) = f(\lambda_1^{\alpha}) = (\chi_{\lambda_1})^{\alpha}$.
Hence  $P \in \nlgc$.

\medskip

Conversely, we  show  that if $P$ is a $\nlgc$-group, then $\zent P$ is cyclic and 
$|P'|= p$.
Let $P$ be a counterexample of minimal order; hence either 
$\zent P$ is not cyclic, or $|P'| > p$.

First, we  suppose that $P' \leq \zent P$ (i.e. that $P$ has nilpotency class $2$).
To begin with, we  also assume that  $Z := \zent P$ is cyclic. So, $|P'| >p$ and, by minimality of $P$, we
have that $|P'| = p^2$. 
Then by ~\cite[Theorem 2.1]{BBC},  $P$ is a central product of $2$-generated subgroups with cyclic center and a (possibly trivial) cyclic subgroup. 
So, again by minimality, we have that $P$ is $2$-generated; 
write $P = \gen{x, y}$.   $P$ being a class $2$ $p$-group, it follows that $\exp{P/Z} = \exp{P'}$. Since $P' \leq Z$ and $Z$ is cyclic, we have that $P'$ is a cyclic group of order $p^2$, and therefore $\exp{P/Z} = \exp{P'} = p^2$.
Thus, again $P$ being a class $2$ $p$-group, we can choose generators $x$ and $y$ of $P$  such that $o(xZ) = o(yZ) = p^2$ in $P/Z$.
Now, if $N$ is the (unique) subgroup of order $p$ of $P'$, then 
both $x^{p}N$ and $y^{p}N$ belong to $M/N := \zent{P/N}$. 
In fact, $P' = \gen{ [x,y]}$  and $[x^{p}, y] = [x, y]^{p} = [x, y^{p}]$, 
so $[x^{p}, y]$ as well as $[y^{p}, x]$ belongs to $N$ as $\exp {P'} = p^2$.
Now, by minimality, $\zent {P/N}$ is  cyclic, and therefore $M/Z$ is cyclic. On the other hand, since both $x^p$ and $y^p$ lie in $M$, $M/Z$ cannot be cyclic. 
This contradiction shows that $Z$ cannot be cyclic.

So,  we assume that $Z$ is not cyclic. 
Let $N$ be a subgroup of order $p$ of $Z$ such that $N \neq P'$ (such a subgroup 
certainly exists as $Z$ has more than one subgroup of order $p$). 
By minimality,  $\zent{P/N}$ is cyclic and $|P'N/N| = p$. 
If $N \cap P' = 1$, then $|P'| = p$ and  by Lemma~\ref{fullyramified} the irreducible characters
of $G$ lying over $1_N \times \lambda$ and $\mu \times \lambda$, where 
$\mu \in \irr N$ and $\lambda \in \irr {P'}$ are non-principal characters, 
are non-linear characters of  the same degree, but with distinct kernels, against 
$P \in \nlgc$. 

Thus, $N \leq P'$,  $|P'| = p^2$ and  $Z = N \times Z_0$, where $Z_0$ is a non-trivial cyclic group. 
Let $M \leq Z_0$ with $|M| = p$. 
By minimality $|(P/M)'| = p$ and  $\zent{P/M}$ is cyclic; this  yields that $M = Z_0$ and hence 
that $Z = N \times M = P'$ is elementary abelian of order $p^2$. 
As $P$ has class two, $\exp{P/Z}  = p$. 
Let $U/N = \zent{P/N}$ and $W/M = \zent{P/M}$; by the minimality of $P$, they  are cyclic groups.
As $\exp{P/Z} = p$,   this yields $|U/Z|, |W/Z| \leq p$.
We claim that $|U/Z| = |W/Z|$.  Contrarily  assume that $|U/Z| \neq |W/Z|$. Then, by symmetry,  let $W = Z$ and $|U/Z| = p$.  So $P/M$ is an extraspecial group and hence
$|P/Z| = p^{2n}$ for some positive integer $n$. 
Then $|P/U| = p^{2n-1}$ is not a square and, as $P/N$ has 
commutator subgroup of prime order, this gives a contradiction by  Lemma~\ref{fullyramified}. This proves our claim. 
Hence, $|P/N: \zent{P/N}| = |P/M: \zent{P/M}|$ and by Lemma~\ref{fullyramified} there are
characters  $\chi, \psi \in \irr{P}$  such that  $\chi(1) = \psi(1)$ with $\ker{\chi} = N$ and $\ker{\psi} = M$,  so they cannot be 
Galois conjugate.


Therefore, we can  assume that $P' \not\leq Z$; hence, in particular, $|P'| \geq p^2$.  
Let $M$ be a normal subgroup of $P$ such that $|M| = p$. 
Since $P/M$ is a non-abelian $\nlgc$-group,  minimality of $P$ yields that
$\zent{G/M}$ is cyclic and that $(P/M)' = P'M/M$ has prime order.
Thus $M  \leq P'$ and, as both $Z/M$ and $P'/M$ are subgroups of the cyclic group $\zent{P/M}$,
and $P'/M$ is a group of prime order which is  not contained in  $Z/M$, 
we deduce that $Z = M$.
So, we conclude that $Z$ is the only normal  subgroup of order $p$ of $P$, $Z < P'$ 
and  that $|P'| = p^2$. 

Now, $\zent{P/Z} = \zentt{P}/Z$ is  a cyclic group of exponent dividing 
$\exp Z = p$ (\cite[III.2.13]{H}). 
As $P'/Z \leq \zent{P/Z}$, we 
conclude that $\zentt{P} = P'$ is a group of order $p^2$. 
Hence, $P/Z$ is an extraspecial group; set $|(P/Z)/(P/Z)'| = |P/P'| = p^{2n}$. 

We remark that $|P| \neq p^4$. In fact, otherwise  $P/Z$ is an extraspecial group of order 
$p^3$, so there exists a character $\chi \in \irr P$
with $\ker{\chi} = Z$ and $\chi(1) = p$. As $|Z| = p$, there also 
exists a faithful character $\psi \in \irr P$ (\cite[(2.32)]{Is}). 
Since  $\psi(1)^2$ divides $|P/Z|$ (\cite[(2.30)]{Is}), 
it follows that $\psi(1) = p = \chi(1)$, a contradiction.

We also observe that $P' \cong C_p \times C_p$.  In fact,  if $P' = \zentt P$ is cyclic then 
$P$ has a cyclic subgroup $C$ of index $2$ by~\cite[III.7.7]{H}. 
Thus $|C/Z| = 4$ (as $P/Z$ is extraspecial). Then $|P/Z| = 2^3$, and hence $|P| = 16$, which is not possible.

Let $N = \cent P{P'}$. Since $P/N$ is isomorphic to a non-trivial subgroup of 
${\rm GL}_2(p)$, then $|P/N| = p$.
Also,  $N' = P'$. In fact, if $N' < P'$, then $N' \leq  Z$, as $Z$ is the only 
proper non-trivial subgroup  of $P'$ which is normal  in $P$ .
Hence, $N/Z$ is an abelian subgroup of index $p$ in the extraspecial group $P/Z$.
This implies $|P| = p^4$, which is again not possible.

Let $T$ be a subgroup of order $p$ of $P'$, with $T \neq Z$. 
Let $\o N  = N/T$ and $\o W = \zent{\o N}$. 
Since $\o{N}' = \o{P'}$ has order $p$, Lemma~\ref{fullyramified} yields that
every $\lambda \in \irr{\o W}$  such that $\o{P'} \not\leq \ker{\lambda}$ is fully ramified in $\o N / \o W$; so, if $\phi \in \irr{N|\lambda}$, then $\phi(1) = p^a$, 
where $|N/W| = p^{2a}$ and $\Q(\phi) = \Q(\lambda)$; note that 
$\Q(\lambda) \subseteq \Q_{p^2}$, as $\exp{W/T}$  divides $p^2$ (since $P/P'$ has 
exponent $p$ because $P/Z$ is extraspecial). 
We also observe that $\phi$ is not $P$-invariant, as $\ker{\phi} \cap P' = T$ is not normal in $P$. Hence, as $|P:N| = p$, $\phi^P$ is an irreducible character of $P$ (\cite[(6.19)]{Is}); since $Z \not\leq \ker{\phi^P}$, then $\ker{\phi^P} = 1$. 
So $\phi^P \in \irr P$ is a faithful character of degree $p^{a+1}$.

Let now $\mu \in \irr {P'}$ with $\ker{\mu} = T$.
As $W/T$ is abelian, it follows that  $\mu$ has $|W/P'|$ extensions to $W$. 
Let $\lambda_1, \lambda_2 \in \irr{W|\mu}$ be any extensions of $\mu$,  $\phi_i \in \irr N$
the (unique)  character lying over $\lambda_i$ and $\chi_i = \phi_i^P \in \irr P$,  for $i = 1,2$.    

As $\chi_1(1) = \chi_2(1) > 1$ and $P \in \nlgc$, 
there exists a Galois automorphism $\alpha \in {\rm Gal}(\Q_{|P|}|\Q)$ 
such that $\chi_1^{\alpha} = \chi_2$. 
Let $\{x_1 = 1, x_2, \ldots, x_p\}$ be a transversal for  $N$ in $P$. 
Then $(\chi_i)_N = \phi_i^{x_1} + \phi_i^{x_2}+ \cdots + \phi_i^{x_p}$, for
$i = 1,2$. 
Hence, $\phi_1^{\alpha} = \phi_2^{x_j}$ for some $j \in \{1, 2, \ldots, p \}$. 
We have 
$$T = \ker{\phi_1} \cap P' = \ker{\phi_1^{\alpha}} \cap P' = \ker{\phi_2^{x_j}} \cap P' =
\left( \ker{\phi_2} \cap P' \right)^{x_j} = T^{x_j}.$$ 
As all conjugates $T^{x_i}$ are distinct (because $|P:N| = p$ and $T$ is not normal in $P$), 
it follows that $j = 1$ and that  $\phi_1^{\alpha} = \phi_2$.
Therefore, recalling that $(\phi_i)_W = \lambda_i$ for $i = 1,2$, we conclude that 
$\lambda_1^{\alpha} = \lambda_2$.
So, by Lemma~\ref{lem1}(a) there exists a $\beta \in \mathcal{H} = {\rm Gal}(\Q(\lambda_1)|\Q)$
such that $\lambda_1^{\beta} = \lambda_2$.
Hence, $\mathcal{H}$ acts transitively on the set $\irr{W|\mu}$ of the extensions of $\mu$ to 
$W$. As $\Q(\lambda_1) \subseteq \Q_{p^2}$, we see that $|\mathcal{H}|$ divides $p(p-1)$.  
Since $|\irr{W|\mu}| = |W/P'|$ is a power of $p$, we conclude that $|W/P'|$ divides $p$. 
Now, $P' < W$, as otherwise $\o N$ would be an extraspecial group 
against $|\o N| = |P/P'| = p^{2n}$.  We conclude that $|W/P'| = p$ and  hence that 
$p^{2a} = |N/W| = p^{2n -2}$. 

Therefore, if $\chi \in \irr P$ lies over $\mu$,  then $\chi(1) = p^{a+1} = p^n$ and, as observed above, $\ker{\chi} = 1$. 
But $p^n$ is also the degree of any non-linear 
irreducible character of the extraspecial group $P/Z$; this is a 
contradiction, as  characters with different kernels cannot be 
Galois conjugate.  
This is the final contradiction.
\end{proof}

As a consequence, we get a characterization of nilpotent groups in $\nlgc$.
\begin{cor}
\label{nilpotent}
  Let $G \in \nlgc$ be a nilpotent group. Then either $G$ is abelian or 
$G$ is a group of prime power order with cyclic center and commutator subgroup of prime order. 
\end{cor}

\begin{proof}
Assume that $G$ is non-abelian and let $P$ be a non-abelian Sylow $p$-subgroup of $G$, 
where $p$ is a suitable prime. So $G = P \times K$ and hence $K = 1$ by part (c) of 
Lemma~\ref{lem1}. 
\end{proof}


Given a finite group $G$, we denote by $G_{\infty}$ the nilpotent residual of 
$G$, that is the smallest term of the lower central series of $G$.

\begin{lem}
  \label{solvable1}
Let $G \in \nlgc$ be a solvable group and let $K = G_{\infty} > 1$ (i.e. $G$ is not nilpotent).
Let  $N < K$ be such  that either $N = 1$ or  $N$ is the unique minimal normal subgroup of $G$ contained in $K$.  
If $(|G/K|,  |K/N|) = 1$,  
then $G$ is a Frobenius group with Frobenius kernel $K$. 
\end{lem}

\begin{proof}
Let $\phi \in \irr K$ with $\phi \neq 1_K$. Then $\phi(1)$ divides $|K/N|$, as $N$ is a normal
abelian subgroup of $K$; so $\phi(1)$ is coprime to $|G/K|$.

We claim that  the determinantal order $o(\phi)$ is also coprime to $|G/K|$.
Recall that $o(\phi)$ divides $|K/K'|$; so we are done if $N \leq K'$. 
But if $N \not\leq K'$, 
then the assumption that $N$ is the unique minimal normal subgroup of $G$ contained in $K$ 
implies that  $K' = 1$ and hence
$K$ is a $q$-group, for some prime $q$. 
So, $o(\phi)$ is a power of $q$ and  $q$ divides $|K/N|$, and the claim follows.

Thus, $(o(\phi)\phi(1), |G/K|) = 1$  and hence (by~\cite[(6.28)]{Is}) there exists a unique extension 
$\alpha$ of $\phi$ to $I_G(\phi)$ such that $o(\alpha) = o(\phi)$. 

Now let  $\beta$ be any extension of $\phi$ to $I = I_G(\phi)$ and let 
$\chi = \alpha^G$ and $\psi = \beta^G$. So, $\chi, \psi \in \irr G$ 
are irreducible  characters of the same degree. 
Observe that they are non-linear,
otherwise their kernels would contain $K$, as $K = G_{\infty} \leq G'$, 
while they lie over $\phi \neq 1_K$.
So, as $G \in \nlgc$, there exists a $\sigma \in Gal(\Q(\chi)|\Q)$ such that $\chi^{\sigma} = \psi$. 
Recalling that Galois conjugation commutes with character induction,  Clifford correspondence 
yields that $\alpha^{\sigma} = \beta$.  In particular 
$o(\beta) = o(\alpha) = o(\phi)$ and hence (by the uniqueness of $\alpha$) we conclude that 
$\beta = \alpha$.
Whence, there exists a unique extension of $\phi$ to $I$.  
Now Gallagher's theorem (\cite[6.17]{Is}) yields 
$(I/K)' = I/K$ and, being $I/K$ solvable, this implies $I = K$. 

Therefore, we have shown that $I_G(\phi) = K$ for every $\phi \in \irr K$, $\phi \neq 1_K$. 
By Brauer Permutation Lemma (\cite[6.32]{Is}), 
we conclude that  no non-trivial  conjugacy class of $K$ is fixed by any  non-trivial element of  $G/K$.
Thus, $\cent Gx \leq K$ for every non-trivial element $x \in K$, so $G$ is a Frobenius group with 
kernel $K$.   
\end{proof}

\begin{pro}
\label{solvable2}      
Let $G$ be a solvable, non-nilpotent group, with $G \in \nlgc$.
Then $G = KL$ is a Frobenius group,  where $K = G_{\infty}$ is the Frobenius
kernel, $L$ is the Frobenius complement and either $L$ is cyclic  or $L \cong Q_8$.  
\end{pro}

\begin{proof}
We first observe that it is enough to show that $G$ is a Frobenius group with kernel $K = G_{\infty}$. 
In fact, $G/K$ is a nilpotent group in $\nlgc$, so Corollary~\ref{nilpotent} and the structure of 
Frobenius complements then yield that $G/K$ is  either a cyclic group  or 
it is a quaternion group $Q_8$,
since $Q_{2^n}$ has commutator subgroup of order $2^{n-2}$.     

We work by induction on $|G|$. Let $K = G_{\infty}$.
Assume first that $K$ is minimal normal in $G$; hence $K$ is an abelian $q$-group for some prime $q$. 

If $K < G'$, then   $G/K \in \nlgc$ is a  non-abelian nilpotent group and 
by Corollary~\ref{nilpotent} $G/K$ is a $p$-group, for a prime $p \neq q$.
Hence, we are done by Lemma~\ref{solvable1} (with $N = 1$).
If $K = G'$, let $Q$ be a Sylow $q$-subgroup of $G$.  
So $K \leq Q$ and $Q$ is normal in $G$. 
Thus, as $1 \neq \zent Q \cap K \nor G$, we see that $K \leq \zent Q$.
Let $L$ be a $q$-complement of $G$. 
Since $[L,Q]$ is a subgroup of $K = G'$, $[L, Q] \nor G$. So $[L, Q] \neq 1$, as otherwise
$G = L \times Q$ would be nilpotent,  and we deduce  that 
$[L, Q] = K$. By coprime action, $Q = [L, Q] \times \cent QL = K \times \cent QL$ and hence $G = LK \times \cent QL$.
Recalling that $LK$ is non-abelian,  
Lemma~\ref{lem1} (c) yields that $\cent QL = 1$.
Therefore, $Q = K$ and we can again apply Lemma~\ref{solvable1}
(with $N = 1$). 

Thus, we can assume that $K$ is not minimal normal in $G$. 
If $N_1, N_2 \leq K$ are distinct minimal normal subgroup of $G$, then by induction 
$G/N_i$ is a Frobenius group with Frobenius kernel $K/N_i$, for $i = 1,2$, 
as $K/N_i = (G/N_i)_{\infty}$. 
In particular, $(|G/K|, |K/N_i|) = 1$ for $i = 1,2$ and hence  $(|G/K|, |K|) = 1$; again, we 
conclude using Lemma~\ref{solvable1}.
So, we can reduce to the case that there is an unique minimal normal subgroup $N$ of $G$ such that 
$N < K$. We conclude by using induction and Lemma~\ref{solvable1}. 
\end{proof}


We now start  working towards a finer description of the solvable (non-nilpotent) 
$\nlgc$-groups. 
In order to motivate the next result, we mention that the affine group 
${\rm GF}(5^2)^+ \rtimes {\rm GF}(5^2)^{\times}$ is a $\nlgc$-group, but its 
subgroup of index $2$ is not a $\nlgc$-group, since it has two \emph{rational} characters of 
degree $12$. 

Given an abelian group $K$, we denote by  $\irr{K}^{\#}$ the set of 
non-principal irreducible characters of $K$. 

\begin{thm}
  \label{AbelianKernel}
Let $G = KL$ be a Frobenius group, with Frobenius kernel $K$  and  complement $L$.
Assume that $K$ is abelian.  
Then $G \in \nlgc$ if and only if $K$ is minimal normal in $G$, $|K| = q^n$ (for a prime $q$) and 
\begin{description}
\item[(a)] either $G \cong (C_3 \times C_3) \rtimes Q_8$ or  
\item[(b)] $L$ is cyclic and $|L| = (q^n -1)/d$, where $d$ is a divisor 
of $q-1$ and $d$ is coprime to $n$.
\end{description}
\end{thm}

\begin{proof}
  Assume $G \in \nlgc$. Let $\phi, \psi \in \irr K$ be non-principal characters; then 
$\phi^G$ and $\psi^G$ are irreducible characters and they have the same degree $|L|$, as
 $K$ is abelian.  Then there exists a Galois automorphism $\alpha$ such that 
$\psi^G = (\phi^G)^{\alpha} = (\phi^{\alpha})^G$ and hence, by Clifford theory, there is an
element $x \in L$ such that $\phi^{\alpha} = \psi^x$. 
In particular, $\ker{\phi} = \ker{\phi^{\alpha}} = \ker{\psi}^x$. 
As every subgroup $N < K$ with cyclic factor group $K/N$, is the kernel of a suitable 
non-principal irreducible character of $K$, it follows that $L$ acts transitively on the 
set of such subgroups. Therefore, $K$ is elementary abelian, as otherwise it has non-trivial cyclic 
factor groups of distinct orders. Also $K$ is an irreducible $L$-module,  because given a
proper non-trivial $L$-submodule $H$ of $K$ there exist $M_1$, $M_2$ maximal subgroups of $K$ 
such that $H \leq M_1$ and $H \not\leq M_2$.    
Hence, $K$ is minimal normal in $G$. Write $|K| = q^n$, where $q$ is a prime.  
 
Viewing $K$ as a ${\rm GF}(q)$-vector space,   $L$ acts transitively on the hyperplanes of $K$ 
so $(q^n-1)/(q-1)$ divides $|L|$. 
Also, as $L$ acts fixed point freely on $K$, we have 
that $|L|$ divides $q^n -1$.

So, if $L \cong Q_8$, then $(q, n) \in \{(3, 2), (7, 2)\}$. 
Assuming $q= 7$, $n = 2$, then $G \cong (C_7 \times C_7) \rtimes Q_8$  
has six irreducible characters of degree $8$, and their fields of values have degree $3$ 
over $\Q$, so $G \not\in \nlgc$. 
Hence, if $L \cong Q_8$, then $G \cong (C_3 \times C_3) \rtimes Q_8$. 
We notice here that  the group $(C_3 \times C_3) \rtimes Q_8$ has exactly two non-linear
irreducible characters, one of degree $2$ and one of degree $8$, so it is a $\nlgc$-group. 

By Proposition~\ref{solvable2}, we can  now assume that $L$ is cyclic. 
So, by Lemma~\ref{3.10} we can assume that $L$ is a subgroup of the multiplicative group $L_0$ of the field 
${\rm GF}(q^n)$ acting on $K = {\rm GF}(q^n)^{+}$. 
We denote by $U_0$ the subgroup of order $q-1$ of $L_0$ and set $U = U_0 \cap L$. 
Write $d = [L_0:L]$, where $d$ divides $q-1$. 

We first show  that $L_0 = LU_0$ if and only if $(d, (q^n-1)/(q-1)) = 1$. 
In fact, assuming $L_0 = LU_0$, we have that $[U_0:U] = [L_0:L] = d$ and hence $|U| = (q-1)/d$.
Since $L_0$ is cyclic, we also have that $|U| = (|U_0|, |L|) = |U|((q-1)/|U|, [L:U])$. 
Hence, $(d, (q^n-1)/(q-1)) = ((q-1)/|U|, [L:U]) = 1$. 
Conversely, as  $[L_0:L] = d$ and $[L_0:U_0] = (q^n -1)/(q-1)$, if $(d, (q^n-1)/(q-1))=1$, 
then $U_0$ and $L$ are subgroups of coprime indices in $L_0$ and hence $L_0 = LU_0$. 

So, observing that (as $d$ is a divisor of $q-1$) $(d, (q^n -1)/(q-1)) = (d, n)$, 
in order to complete the proof of the theorem it is enough to show that $G \in \nlgc$ if and only if 
$L_0 = LU_0$. 

We  notice that, identifying $U_0$ with the group of non-zero residue classes $\o a$ $\mod q$, 
if $\o a \in U_0$ and $\lambda \in \irr K^{\#}$, then $\lambda^{\o a} = \lambda^{\alpha_a}$,
where $\alpha_a \in \mathcal{G}_q = {\rm Gal}(\Q_q|\Q)$ is the Galois automorphism that takes $q$-th roots of
unity to their $b$-th power, where  $b$ is the inverse of $a$
$\mod q$. 

Assume first that $L_0 = LU_0$ and take non-linear characters $\chi, \psi \in \irr G$. 
Then  $\chi = \lambda^G$ and $\psi = \mu^G$ for some $\lambda, \mu \in \irr K^{\#}$. 
As $L_0 = LU_0$ acts transitively on $\irr K^{\#}$, there exists a Galois automorphism  
$\alpha \in \mathcal{G}_q $ and an element $x \in L$ such that
$\lambda = (\mu^x)^{\alpha}$. 
Hence, 
$$\chi = \lambda^G = ((\mu^x)^{\alpha})^G = ((\mu^x)^G)^{\alpha} = \psi^{\alpha}$$
because $(\mu^x)^G = \mu^G$ as $x \in L$.
So $G$ is a $\nlgc$-group.

Assume now that $G \in \nlgc$.   
Note that $G$ has exactly $d$ non-linear
irreducible characters, all of degree $|L|$, because $|\irr{K}^{\#}|/|L| = d$ 
and $G = KL$ is a Frobenius group. 
Hence, considering a non-linear $\chi \in \irr G$, 
we have that $d$ divides $|{\rm Gal}(\Q(\chi)|\Q)| = [\Q(\chi):\Q]$. 
Writing $\chi = \lambda^G$ for a suitable $\lambda \in \irr K^{\#}$, one observes that
$\chi(g) = 0$ for every $g \in G \setminus K$. 
For $x \in K$, taking a transversal $T$ of $U$ in $L$, one has 
$$\chi(x) = \sum_{y \in L} \lambda^y(x) =   
\sum_{t \in T} \sum_{a \in U} \lambda^{at}(x) =   
\sum_{t \in T} \sum_{a \in U} \lambda^a(x^{t^{-1}}) =   
\sum_{t \in T} \sum_{\alpha \in \mathcal{H} } (\lambda(x^{t^{-1}}))^{\alpha}$$
where $\mathcal{H}$ is a subgroup of $\mathcal{G}_q$ such that 
$[\mathcal{G}_q :\mathcal{H}] = [U_0:U]$.  
Hence $\sum_{\alpha \in \mathcal{H} } (\lambda(x^{t^{-1}}))^{\alpha}$
 is an element of $E = \rm{Fix}(\mathcal{H})$. 
We conclude that $\chi(x)\in E$ for every $x \in K$ and then $\Q(\chi) \subseteq E$. 
Therefore,  $d$ divides $[E:Q] = [U_0:U] = [LU_0:L]$ and hence $L_0 = LU_0$. 
\end{proof}

A non-abelian $2$-group $K$ is a \emph{Suzuki $2$-group} if
$K$ has more than one involution and there exists a soluble group of automorphisms of $K$ 
which is transitive on the set of the involutions of $K$ (see~\cite[Definition 7.1]{HuBl}). 
If $K$ is a Suzuki $2$-group, then 
$K' = \zent K = \Phi(K) = \{ x\in K \mid x^2 = 1\}$ and
either $|K| = |K'|^2$ or $|K| = |K'|^3$ (\cite[Theorem 7.9]{HuBl}).  

\begin{thm}
  \label{solvable5}
Let $G = KL \in \nlgc$ be a solvable Frobenius group with kernel $K = G_{\infty}$ and complement $L$. 
Then either $K$ is an elementary abelian $q$-group, $q$ a prime,  or $K$ is a Suzuki $2$-group such that $|K| = |K'|^2$, $L$ is cyclic and $|L| = |K'|-1$.    
\end{thm}

\begin{proof}
By Proposition~\ref{solvable2} we know that  $L$ is  cyclic or  $L \cong Q_8$.
So, if  $K$ is abelian, we conclude by applying  Theorem~\ref{AbelianKernel}.

We assume now that $K$ is non-abelian. Thus $L$ is cyclic (as Frobenius groups with 
complements of even order have abelian kernel). 
By applying Theorem~\ref{AbelianKernel} to the $\nlgc$-group $G/K'$, 
we get that $K/K'$ is an irreducible $L$-module; write
$|K/K'| = q^n$, for a prime $q$. 
Since $K$ is nilpotent, this implies  that $K$ is a $q$-group. 
Also, $|L| = (q^n-1)/d$, where $d$ divides $q-1$ and $(d,n)=1$.

We will first show that $q=2$. 
By taking a suitable factor group, we can assume that $K'$ is minimal normal in $G$. 
Hence,  $\exp{K} \in \{q, q^2\}$ and $K' = \zent K$. Moreover, by Lemma~\ref{3.10} we have that 
$|K'| = |K/K'|$.
If $\exp{K} = q$, then by Lemma~\ref{2.6} we see that every non-principal irreducible character of $K'$ is 
fully ramified in $K/K'$. Hence, by the second orthogonality relation, for any
$x \in K \setminus K'$ we have  
$|K/K'| =  |\cent{K/K'}{xK'}| = |\cent K x| \geq |\langle x, K' \rangle| > |K'|$,
a contradiction. 
So, $\exp{K} = q^2$. 

Assume, working by contradiction, that $q \neq 2$. So $K$ (having class $2$) is a regular $q$-group
and $K' = \Omega_1(K) = \{ x \in K \mid x^q = 1 \}$. 
Looking at the action of $L$ on $K'$, by Lemma~\ref{3.10} we can identify  $K'$
with the additive group of the field $\mathbb{F} = {\rm GF}(q^n)$ and $L$ 
with a subgroup of $M = \mathbb{F}^{\times}$. Let $U = {\rm GF}(q)^{\times} \leq M$. 
As $|L| = (q^n-1)/d$, with $d$ a divisor of $q-1$ and $d$ coprime to $n$, 
we have that $LU = M$; in fact, $(|M:L|, |M:U|) = (d, (q^n-1)/(q-1)) = (q, n) = 1$. 
It follows that $L$ acts transitively on the subgroups of order $q$ of 
$K'$.
As all elements in $K \setminus K'$ have order $q^2$, we conclude  $L$ acts transitively on the subgroups of order $q$ of $K$ and 
hence  Shult's theorem~\cite{S} yields that $K$ is abelian, a contradiction. 
Hence, $q=2$ and  then $|L| = 2^n -1$.

We are going to show  that   $K'$ is   minimal normal in $G$.
Assume,working by contradiction, that  there exists a non-trivial subgroup $N$ 
of $K'$ such that $K'/N$ is a chief factor of $G$. 
Taking a suitable factor group, we can also assume that $N$ is minimal normal in $G$.  So, $N \leq \zent K$ and hence $N$ is an irreducible $L$-module. 
By Lemma~\ref{3.10}, then   
$K/K'$, $K'/N$ and $N$ are all faithful irreducible
$L$-modules of the same order  $2^n$, for a positive integer $n$ .
Moreover, $K'$ is abelian, as $[K, K', K] = [K', K, K] = 1$ implies $[K', K'] = 1$ by the Three Subgroups Lemma. 
Finally, we note that by induction 
$K/N$ is a Suzuki $2$-group with $|K/N| = |K'/N|^2$. 

Now, $\exp{K'} \in \{2, 4\}$. If $\exp{K'} = 4$, then $N \setminus \{ 1 \}$
is the set of all involutions of $K'$ (as $K'/N$ is an irreducible $L$-module).
As in the Suzuki group $K/N$ all elements not belonging to $K'/N$ have order $4$, 
we see that $N \setminus \{ 1 \}$ is the set of all involutions of $K$. 
Observing that $L$ acts transitively on $N \setminus \{ 1 \}$, we conclude that
$K$ is a Suzuki $2$-group. 
So, in particular $|K'| \leq |K|^{1/2}$, a contradiction as $|K'| = |K|^{2/3}$. 
Thus, $K'$ is elementary abelian and then $\exp K = 4$. 

Consider now a non-principal character $\lambda \in \irr N$. 
Let $\mu \in \irr{K'|\lambda}$ and let $T = I_K(\mu)$. 
By~\cite[Lemma 12.6]{MW} there exists a (uniquely determined) subgroup $U = U_{\mu}$, with $K' \leq U \leq T$, 
such that every $\nu \in \irr{U|\mu}$ extends $\mu$ and is fully ramified 
in $T/U$ (so, in particular, $|T/U|$ is a square). 
By Clifford correspondence, it follows that all characters $\theta\in \irr{K|\mu}$ 
have the same degree (depending only on $|U/K'|$) and that 
$|\irr{K|\mu}| = |U/K'|$. 
By Lemma~\ref{2.6} it follows that $|U/K'| \leq 2$.   

For  $\mu_0 \in \irr{K'|\lambda}$, we have 
 $\mu_0 = \mu\epsilon$, for some $\epsilon \in \irr{K'/N}$
and $K'/N$ is central in $K/N$, so  $I_K(\mu_0) = T$. 
Recalling that $|T:U_{\mu_0}|$ is a square and that 
both $|U/K'|$ and $|U_{\mu_0}/K'|$ are at most $2$, we also get that
$|U| = |U_{\mu_0}|$. Hence, all characters $\theta \in \irr{K|\lambda}$ have the same degree. 
So, again using  Lemma~\ref{2.6},  we get that $|\irr{K|\lambda}| \leq 2$.

Observing that $|\irr{K'|\lambda}| = |K'/N| = |K/K'|$, we deduce that  
the number $m$ of orbits in the conjugation action 
of $K/K'$ on $\irr{K'|\lambda}$ is $|T/K'|$.
But, by Clifford theorem, $m \leq |\irr{K|\lambda}|$, so we get that 
$|T/K'|  \leq 2$. 
As $|T:U|$ is a square and $K' \leq U \leq T$, we conclude that   $T = U$. 
Assume $|T/K'| = 2$ and let $\mu_1, \mu_2 \in \irr{K'|\lambda}$ be representatives of
the orbits of $K/K'$ on $\irr{K'|\lambda}$. 
As observed before, $|\irr{K|\mu_i}| = |U/K'| = 2$. 
But by Clifford theorem $\irr{K|\mu_1}  \cap \irr{K|\mu_2} = \emptyset$, and 
$\irr{K|\lambda} =  \irr{K|\mu_1}  \cup \irr{K|\mu_2}$, hence 
$|\irr{K|\lambda}| = 4$, a contradiction.  

It follows that $T = K'$ and hence $\mu^G  = \theta \in \irr K$ for every $\mu \in \irr{K'|\lambda}$.
Since $K/K'$ is transitive on $\irr{K'|\lambda}$, then $|\irr{K|\lambda}| = 1$ and we conclude that every non-principal $\lambda \in \irr N$
is fully ramified with respect to $K/N$. 
Therefore,  we have that 
$\chi(x) = 0$ for all $\chi\in \irr K \setminus \irr{K/N}$ and $x \in K \setminus N$. 
Now, $K/N$ has exponent greater than $2$ (as $K/N$ is non-abelian) and hence there
exists an element $y \in K \setminus N$ such that $x := y^2 \not\in N$. 
Since both $K/K'$ and $K'/N$ are elementary abelian, it follows that $x  \in K'$
and that $y \in K \setminus K'$. 
By the second orthogonality relation, $|\cent K x| = |\cent{K/N}{xN}|$. 
But $|\cent{K/N}{xN}| = |K/N| = |N|^2 = |K'|$, while both $y$ and $K'$ centralize $x$
so $|\cent K x| > |K'|$, a contradiction.     

\medskip
Hence, we have that $K'$ is minimal normal in $G$. So, $K' \leq \zent K$
and both $K/K'$ and $K'$ are irreducible $L$-module. In particular, 
$K' = \zent K$ and $|K'| = |K/K'| = 2^n$.  
Recalling that $|L| = 2^n -1$, we see that  
$L$ acts transitively on the non-identity
elements of both $Z = \zent K$ and $K/Z$. 
As all elements in a coset $xZ$, with $x \in K \setminus Z$,  have the same
order, we deduce that all elements $x \in K \setminus Z$ have order $4$. 
Therefore, $Z \setminus \{1 \}$ is the set of the involutions of $K$. 
As $L$ acts transitively on it,
we conclude that $K$ is a Suzuki $2$-group with $|K| = |K'|^2$. 
The proof is complete. 
\end{proof}


\medskip

As defined in~\cite{DNT} a finite group $G$ is a $\gc$-group 
(or $G \in \gc$, for short)
if every two \emph{non-principal} irreducible characters of $G$ are Galois conjugate whenever
they have the same degree. 
Clearly,  $\gc$ is a subclass of $\nlgc$ and, 
for a perfect group $G$, $G \in \nlgc$ if and only if $G \in \gc$.   
We are going to show that a non-solvable $\nlgc$-group is perfect,
and then we  apply the classification of non-solvable $\gc$-groups given 
in~\cite{DNT}.

  Let  us consider the following list of simple groups: 
$$\mathcal{S} =  \{ \AAA_5,Sz(8),J_2,J_3,L_3(2),M_{22},Ru,Th,{}^3D_4(2)\} \; .$$

As proved in~\cite{DNT}, all  groups in $\mathcal{S}$
are $\gc$-groups and hence (being perfect) are $\nlgc$-groups. 

We will make use of the following result from~\cite{DNT}:

\begin{thm}[\mbox{\cite[Theorem 4.1]{DNT}}]
\label{DNT4.1}
  Let $S$ be a non-abelian simple group. Then either $G \in \mathcal{S}$ or
for every almost simple group $A$ with socle $S$ there exist two non Galois conjugate characters  $\chi_1, \chi_2 \in \irr A$
such that $\chi_1(1) = \chi_2(1) >1$ and $(\chi_i)_S \in \irr S$ for $i = 1,2$.
\end{thm}

The proof of the following result mimics the proof of~\cite[Theorem 3.2]{DNT}; we have
anyway decided to sketch it for completeness. 

\begin{lem}
\label{DNT3.2}
  If $G \in \nlgc$ and $S$ is a non-abelian composition factor of $G$, then 
$S \in \mathcal{S}$. 
\end{lem}

\begin{proof}
Let $S$ be a non-abelian composition factor of the $\nlgc$-group $G$. 
Then $G$ has a chief factor $N/M \cong S^n$, for some positive integer $n$. 
By replacing $G$ with a suitable factor group, we can assume that 
$N = S_1 \times S_2 \times \cdots \times S_n$, with $S_i \cong S$, is a minimal normal subgroup
of $G$ and that $\cent GN = 1$. Set $S = S_1$ and  
write $T = \norm G{S}$, $C = \cent G{S}$; so $T/C$ is an almost simple group with socle
(isomorphic to)  $S$. 

Assume, working by contradiction, that $S \not\in \mathcal{S}$; so by Theorem~\ref{DNT4.1}
there exist two non Galois conjugate characters $\theta_1, \theta_2 \in \irr{T/C}$ such that 
$\alpha_i = (\theta_i)_{S} \in \irr{S}$, for $i =1,2$.  

For $i =1,2$, let  
$$\beta_i = \alpha_i \times 1_{S_2} \times \cdots \times 1_{S_n} \in \irr N$$
and observe that $I_{G}(\beta_i) = T$. 

Considering now $\theta_i \in \irr T$ by inflation, we have that $(\theta_i)_N = \beta_i$. 
Hence, by Clifford correspondence $\chi_i = (\theta_i)^G \in \irr G$, for $i=1,2$. 
As $\chi_1(1) = \chi_2(1) >1$, there exists a Galois automorphism $\sigma \in \mathcal{G}_{|G|}$
such that $\chi_1^{\sigma} = \chi_2$. 

But, for $i = 1, 2$ 
$(\chi_i)_N = \sum_{j=1}^n \beta_i^{x_j}$, where $\{x_1 = 1, x_2, \ldots, x_n\}$ is a transversal
of $T$ in $G$. 
As Galois conjugation commutes with induction and restriction, we get that 
$\beta_1^{\sigma} = \beta_2^{x_j}$ for some $j$ and, as 
$\ker{\beta_1^{\sigma}} = \ker{\beta_1} = S_2 \times \cdots \times S_n \neq \ker{\beta_2^{x_j}}$ 
for $j >1$, we conclude
that $\beta_1^{\sigma} = \beta_2$.
So,  $\theta_1^{\sigma}, \theta_2 \in \irr{T|\beta_2}$ and hence Clifford correspondence yields
that $\theta_1^{\sigma} =  \theta_2$, against the fact that $\theta_1$ and $\theta_2$ are not
Galois conjugate. 
\end{proof}

\begin{thm}
\label{perfect}
Let $G$ be a non-solvable group with $G \in \nlgc$. Then $G = G'$ 
\end{thm}

\begin{proof}
We work by induction on $|G|$. 
Assume that $G$ is not a simple group and let $N$ be
 a minimal normal subgroup of $G$. 

Suppose first that $G/N$ is solvable. Then $N$ is non-solvable and hence
$N = S^k$ for some $S \in \mathcal{S}$ by Proposition~\ref{DNT3.2}.
Then $S$ has a non-principal $\Aut(S)$-invariant rational character $\alpha$ of
odd degree (see~\cite[page 299]{DNT} or~\cite{Atlas}).  
Then $\psi = \alpha \times \alpha \times \cdots \times \alpha \in \irr N$ 
is a $G$-invariant rational character of odd degree with $o(\psi) = 1$
(as $N = N'$). 
By~\cite[Theorem 2.3]{NT} it follows that $\psi$ extends to a rational character
$\chi\in\irr G$. As we are assuming that $G/N$ is a non-trivial solvable group, 
there exists a non-principal linear character $\lambda \in \irr G$. 
So, $\lambda\chi \in \irr G$ is a non-linear character of the same degree as $\chi$,
and $\lambda\chi \neq \chi$ by Gallagher's theorem. 
Being $G$ a $\nlgc$-group, this gives a contradiction, as $\chi$ is rational, so it is fixed by Galois conjugation.

Hence,  $G/N$ is non-solvable. By induction, $G/N$ is perfect and hence $G'N = G$. 
Assuming that $N$ is not contained in $G'$, we have  
$N \cap G'  = 1$  and hence $ G = G' \times N$. 
So $N$ is a non-trivial abelian group. As $G'$ is non-abelian (since $G$ is non-solvable), 
Lemma~\ref{lem1} (c) yields that $N  = N' = 1$, a contradiction. 
Therefore,  $N \leq G'$ and hence $G = G'N = G'$.  
\end{proof}
%
%
We are now ready to prove Theorem~A. 
\begin{proof}[Proof of Theorem~A]
Assume first that $G$ is a non-abelian  solvable $\nlgc$-group. 
If $G$ is nilpotent, then we get  type $(a)$. 
Assume that $K = G_{\infty} < G$. Then by Proposition~\ref{solvable2} $G$ is a Frobenius group
with Frobenius kernel $K$ and complement $L$, with $L$ cyclic or $L \cong Q_8$ (so $G$ is 
of type $(b)$).
If $K$ is abelian, then we get either type $(b1)$ or $(b2)$ by Theorem~\ref{AbelianKernel}. 
If $K$ is non-abelian, then we get type $(b3)$ by Theorem~\ref{solvable5}.   

Assume now that $G$ is a non-solvable $\nlgc$-group. Then  Theorem~\ref{perfect} 
yields that $G = G'$ and hence $G$ is a $\gc$-group. Now $(c)$ follows by Theorem~A of~\cite{DNT}.
 
\medskip
Conversely, we will now show that any group of type $(a)-(c)$ is a $\nlgc$-group. 
Groups of type $(a)$ are $\nlgc$-groups by Theorem~\ref{p-groups}.

Let now $G$ be a group of type $(b)$. If it is of type $(b1)$, then $G$ has only two non-linear
irreducible characters, one of degree $2$ and the other of degree $8$, so it is 
a $\nlgc$-group. 
If $G$ is of type $(b2)$, then $G$ is a $\nlgc$-group by Theorem~\ref{AbelianKernel}. 

So we assume that $G$ is of type $(b3)$. 
We note that all non-linear characters $\chi \in \irr G$ of odd order have the 
same degree $|L|$ and that they are Galois conjugate. In fact, $\ker{\chi} = K'$ and
$G/K'$ is a $\nlgc$-group by Theorem~\ref{AbelianKernel}. 
Now, denoting by $\Delta$ the set of all non-linear irreducible  characters of $K$,   by  
Theorem~7.9 of~\cite{HuBl} and Lemma~2.9 of~\cite{DNT} we have that $|\Delta| = 2|L|$
and that   every $\theta \in \Delta$ is not rational valued.
Hence, no $\theta \in \Delta$ is
real valued,  as $\theta_{K'}$ is rational valued 
and every  element in  $K \setminus K'$ has order $4$.  
Considering that by Brauer Permutation Lemma $L$ acts fixed point freely on 
$\irr K \setminus \{ 1_K \}$, we see that  $L$ has exactly two orbits $O_1$ and $O_2$ on $\Delta$.
Write $\chi = \theta_1^G$, $\psi =\theta_2^G$, where  $\theta_1, \theta_2 \in \irr K$ non-linear
characters,   
and $\theta_1 \in O_1$ and $\theta_2 \in O_2$. 
Now,  $\overline{\theta_1} = \theta_2^x$ for some $x \in L$, because 
complex conjugation does not stabilize the orbit $O_i$, for $i = 1,2$ 
(otherwise, as $|O_i|$ is odd, there would be some real character in $O_i$). 
Hence, 
$$\overline{\chi} = \overline{\theta_1^G} = (\overline{\theta_1})^G = (\theta_2^x)^G = \psi \; .$$
Thus $G$ is a $\nlgc$-group.
 
(We also remark that in this case $(b3)$ $G$ has exactly three non-linear
irreducible characters, one of odd degree $|L|$ and $\chi$ and $\psi$ above.)

Finally, by Theorem~A of~\cite{DNT} the  groups listed 
in $(c)$ are $\gc$-groups. So, being perfect groups,  they are also 
$\nlgc$-groups. 
\end{proof}

To conclude, we prove Corollary~B, the Berkovich-Chillag-Herzog classification of 
groups with distinct non-linear degrees. 

\begin{proof}[Proof of Corollary B]
It is enough to check which of the groups listed in Theorem~A have distinct non-linear degrees. 

For a $p$-group $G$ of type $(a)$, Lemma~\ref{fullyramified} yields that $G$ has 
exactly $(p-1)|\zent G : G'|$ non-linear irreducible characters of the same degree. 
Hence, $G$ has distinct non-linear degrees if and only if $p=2$ and $G$ is extraspecial. 

Assume now that $G$ is of type $(b)$, so $G = KL$ is a Frobenius group. 
First, we recall that $(C_3 \times C_3) \rtimes Q_8$ has distinct non-linear degrees (exactly 
one irreducible character  of degree $2$ and of degree $8$). 
Next, we observe that the groups listed in type $(b2)$ have exactly $d$ non-linear characters,
all of degree $|L|$. Hence they are groups with distinct non-linear degrees if and only if 
$d=1$. 
Also, by the remark at the end of the proof of Theorem~A, the groups of type 
$(b3)$ have
two non-linear irreducible characters of the same degree.     

Finally, it is readily checked (see~\cite{Atlas}) that the groups of type (d) also have
two non-linear irreducible characters of the same degree.     
\end{proof}



\begin{thebibliography}{ABCD}

\bibitem[BCH]{BCH} Y. Berkovich, D. Chillag and M. Herzog, Finite groups
in which the degrees of the nonlinear irreducible characters are distinct, \textit{Proc. Amer. Math. Soc.} {\bf 115 } (1992), 955--959.

\bibitem[BBC]{BBC} J. M. Brady, R. A. Bryce and J. Cossey, 
On certain abelian-by-nilpotent groups, {\it Bull. Austral. Math. Soc.} {\bf 1} (1969), 403-416.

\bibitem[Atlas]{Atlas}
 J. H. Conway, R. T. Curtis, S. P. Norton, R. A. Parker and R. A. Wilson,
`{\it An ATLAS of Finite Groups}', Clarendon Press, Oxford, $1985$.


\bibitem[DNT]{DNT} S. Dolfi, G. Navarro and P.H. Tiep, 
Finite groups whose same degree characters are Galois conjugate, 
\textit{Israel J. Math.} {\bf 198} (2013), 283--331. 



\bibitem[H]{H} B. Huppert, {\it Endliche Gruppen I }
Springer, Berlin, 1983.


\bibitem[HB]{HuBl} B. Huppert, N. Blackburn, `{\it Finite Groups II}',   
Springer, Berlin, 1982.

\bibitem[H1]{Hu1} B. Huppert, {\it Character Theory of Finite Groups }
De Gruyter, Berlin-New York, 1998.


\bibitem[I]{Is}
 I. M. Isaacs, {\it Character Theory of Finite Groups}, Dover, New York,
1994.


\bibitem[MW]{MW} O. Manz and T. Wolf, Representations of solvable groups, 
London Math. Soc. Lect. Note Series {\bf 185}, Cambridge Univ. Press, 1993.

\bibitem[NT]{NT}
 G. Navarro and Pham Huu Tiep, Rational irreducible characters and
rational conjugacy classes in finite groups,
{\it Trans. Amer. Math. Soc.} {\bf 360} (2008), 2443--2465.


\bibitem[S]{S} E. Shult, On finite automorphic algebras,
\textit{Illinois J. Math.}  {\bf 13}  (1969),  625--653.

\end{thebibliography}
\end{document}